\documentclass{amsart}

\usepackage{amsmath,amsthm,amssymb,amscd,enumerate,mathtools,enumitem,tikz-cd,bm}
\usepackage{amsfonts}
\usepackage{rotating}
\usepackage{euscript}
\usepackage{pst-node}
\usepackage{epsfig,verbatim}

\def\be{\begin{equation}}
\def\ee{\end{equation}}

\def\C{{\mathbb C}} 
\def\f{\EuScript}
 
\def\P{{\mathbb P}}
\def\Z{{\mathbb Z}}

\def\phi{{\varphi}}

\def\deg{{\rm deg\,}}

\def\Aut{{\rm Aut}}
\def\Mon{{\rm Mon}}
\def\mult{{\rm mult}}

\def\bp{\begin{proposition}}
\def\ep{\end{proposition}}

\def\bt{\begin{theorem}}
\def\et{\end{theorem}}
\def\br{\begin{remark}}
\def\er{\end{remark}}
\def\be{\begin{equation}}
\def\bee{\begin{equation*}}
\def\l{\label}

\def\ee{\end{equation}}
\def\eee{\end{equation*}}
\def\bl{\begin{lemma}}
\def\el{\end{lemma}}
\def\bc{\begin{corollary}}
\def\ec{\end{corollary}}
\def\pr{\noindent{\it Proof. }}

\def\bd{\begin{definition}}
\def\ed{\end{definition}}
\def\t{\widetilde}

\def\hat{\widehat}

\def\t{\widetilde }

\newtheorem{theorem}{Theorem}[section]
\newtheorem{lemma}[theorem]{Lemma}
\newtheorem{definition}[theorem]{Definition}
\newtheorem{corollary}[theorem]{Corollary}
\newtheorem{proposition}[theorem]{Proposition}
\newtheorem{problem}[theorem]{Problem}

\theoremstyle{definition}

\theoremstyle{definition}
\newtheorem{remark}[theorem]{Remark}
\def\bpr{\begin{problem}}
\def\epr{\end{problem}}

\mathtoolsset{showonlyrefs}

\begin{document}

\title[]{On intersections of function fields
}

\author[F. Pakovich]{Fedor Pakovich}
\address{Department of Mathematics, 
Ben Gurion University of the Negev, P.O.B. 653, Beer Sheva,  8410501, Israel}
\email{pakovich@math.bgu.ac.il}

\begin{abstract} Let $
X: \f E \rightarrow \f R $ and $ Y: \f E \rightarrow \f T
$ be  holomorphic maps 
 between compact Riemann surfaces 
such that 
$
X^* M(\f R) \cdot Y^* M(\f T) = M(\f E),
$ where $M(\f S)$ denotes the field of meromorphic functions
on a Riemann surface $\f S$.
We study the problem of determining when the degree
$
[ M(\f E) : X^* M(\f R) \cap Y^* M(\f T)] 
$
is finite and attains the minimal possible value $\deg X\cdot\deg Y$.
When $X$ is a Galois covering, we give a complete characterization
in terms of factorizations of $Y$ into an equivariant map
followed by a Galois covering.
We also establish some results without the maximality condition,
as well as general results on the functional equation
$A\circ X=B\circ Y$ when one of the maps involved is a Galois covering.

\end{abstract}

\maketitle

\section{Introduction} 
Let $X:\f E\rightarrow\f R$ and $Y:\f E\rightarrow\f T$
be holomorphic maps between compact Riemann surfaces. We will assume that they 
satisfy the condition
\begin{equation}\label{us1}
X^* M(\f R) \cdot Y^* M(\f T) = M(\f E),
\end{equation} 
where $M(\f S)$ denotes the field of meromorphic functions
on a Riemann surface $\f S$.
Equivalently, $X$ and $Y$ cannot be written as
\[
X=X_1\circ W,\qquad Y=Y_1\circ W,
\]
where $X_1,Y_1,W$ are holomorphic maps between compact
Riemann surfaces and $\deg W>1$. 

In this paper, we study the following fundamental problem:
under what conditions on $X$ and $Y$ does
\begin{equation}\label{us2}
X^*M(\f R)\cap Y^*M(\f T)\neq\mathbb C
\end{equation} 
hold? Equivalently, when do there exist holomorphic maps
between compact Riemann surfaces
$A:\f R\rightarrow\f C$ and $B:\f T\rightarrow\f C$
such that the diagram
\be \l{sas}
\begin{CD}
\f E @>Y>> \f T\\
@VV X V @VV B V\\
\f R @>A>> \f C
\end{CD}
\ee
commutes?  

In this case,
\be \l{sos}
X^*M(\f R)\cap Y^*M(\f T)=H^*M(\f C),
\ee
for some holomorphic map $H:\f E\rightarrow\f C$, 
and there exist holomorphic maps 
$A:\f R\rightarrow\f C$ and $B:\f T\rightarrow\f C$
satisfying
\begin{equation}\label{h}
H=A\circ X=B\circ Y.
\end{equation}
We will call a map $H$ satisfying \eqref{sos} a
{\it generator} of $X^*M(\f R)\cap Y^*M(\f T)$.
On the other hand, we refer to any nonconstant
holomorphic map $H$ satisfying \eqref{h} for some
holomorphic maps $A$ and $B$ as a
{\it common composite} of $X$ and $Y$.

We also consider a version of the above problem in which
the intersection of fields $X^*M(\f R)\cap Y^*M(\f T)$
is maximal in $M(\f E)$, in the sense that
the extension degree attains its minimal possible value,
that is,
\begin{equation}\label{us3}
[M(\f E):X^*M(\f R)\cap Y^*M(\f T)]
=\deg X\cdot\deg Y.
\end{equation} 
We call these problems the {\it intersection problem}
and the {\it maximal intersection problem}, respectively.

Despite the fundamental nature of the intersection problem, we are aware of only two
papers devoted specifically to it, namely those by 
Accola  \cite{ac} and Bogomolov and Tschinkel \cite{bt}.   
Accola \cite{ac} showed that, under condition~\eqref{us1},
equality~\eqref{us3} holds if and only if 
the correspondences defined by the curves
$$
X(x)=X(y)
\qquad {\rm and} \qquad
Y(x)=Y(y)
$$
commute.
He also obtained some results under additional
hypotheses in the case of equality in the
Castelnuovo--Severi inequality.

Bogomolov and Tschinkel \cite{bt} considered  {\it unramified} coverings $X:\f E\rightarrow\f R$ and $Y:\f E\rightarrow\f T$
of surfaces of genus at least two. 
In this setting, taking the universal covering
$\pi:\mathbb H\to\f E$, we obtain two universal coverings
$$
X\circ\pi:\mathbb H\longrightarrow\f R,\qquad
Y\circ\pi:\mathbb H\longrightarrow\f T,
$$
which allows one to work with Fuchsian groups acting
on the same upper half-plane. 
Using Margulis's commensurator theorem, the authors of \cite{bt} showed that $X$ and $Y$ always have a common composite if the uniformizing Fuchsian groups are non-arithmetic. 
In the arithmetic case, however, they constructed examples
in which the pulled-back function fields intersect only
in the constants.

The intersection problem remains largely open even
in the case $\f E=\C\P^1$, which seems particularly
interesting because the maps $X$ and $Y$ are then
rational functions and  can be written explicitly   in terms of the coefficients of their numerators and denominators. 
Notice that in this case condition \eqref{us1} reduces to
\begin{equation}\label{us11}
\C(X,Y)=\C(z),
\end{equation}
while condition \eqref{us2} reduces to
\begin{equation}\label{us21}
\C(X)\cap\C(Y)\neq\C, 
\end{equation}
which is equivalent to the existence of
non-constant rational functions $A$, $B$, and $H$  satisfying \eqref{h}.

Essentially, the only broad setting in which the
intersection problem for rational functions admits
a complete solution is the case where both
$X$ and $Y$ are {\it polynomials}.
In this case, condition \eqref{us21} implies that
the generator $H$ of the intersection $\C(X)\cap\C(Y)$
can be chosen to be a polynomial, and Ritt's theory
of polynomial decompositions (see \cite{r1}, \cite{sch}) implies that polynomials
$X$ and $Y$ satisfy conditions \eqref{us11} and
\eqref{us21} if and only if, after possibly interchanging
$X$ and $Y$ and applying affine changes of coordinates
in the source and targets, either
\be \l{fde0}
X=z^n,\quad Y=z^sR(z^n),
\ee
where $R$ is a polynomial, $n\geq1$, $s\geq0$, and
$\gcd(s,n)=1$, or
\be \l{che}
X=T_n,\quad Y=T_m,
\ee
where $T_n$ and $T_m$ are Chebyshev polynomials,
$n,m\geq1$, and $\gcd(n,m)=1$. 
For such $X$ and $Y$, a common composite is obtained
by taking 
\begin{equation}\label{fde}
A=z^sR^n(z),\quad B=z^n
\end{equation}
in the first case and
\[
A=T_m,\quad B=T_n
\]
in the second.
In both cases, the common composite has degree
$\deg X\cdot\deg Y$. Thus, in the polynomial case,
conditions \eqref{us1} and \eqref{us2} imply
condition \eqref{us3}.

Note that some results concerning equation \eqref{h}
with $X$ and $Y$ being polynomials are also known
in settings where $A$ and $B$ are allowed to belong
to broader classes of functions.
Specifically, in \cite{lys}, a description of the
solutions of \eqref{h} was obtained in the case where
$A$ and $B$ are transcendental meromorphic functions
in $\mathbb C$. Some results in this direction were
also given in \cite{hs}.
In \cite{pcon}, equation \eqref{h} was solved under
the assumption that $A$ and $B$ are continuous functions
on $\mathbb C\mathbb P^1$ (in particular, $A$ and $B$ may be rational functions).
Finally, in \cite{form}, equation \eqref{h} was studied
in the case where $A$, $B$ and $X$, $Y$ are formal
power series of order at least two, that is, series
of the form $\sum_{n\geq2}c_nz^n$.

The relation with equation \eqref{h} shows that 
the intersection problem for rational functions is, in a sense, dual to the problem
of describing ``separated-variable'' curves
\begin{equation}\label{cur}
A(x)-B(y)=0,
\end{equation}
having an irreducible component of genus zero, where $A$
and $B$ are rational functions.
The latter problem is connected with many areas of mathematics
and has been studied intensively
(see, e.g., \cite{bkn}, \cite{bil}, \cite{fc}, \cite{alg}, \cite{plo}
and the references therein). 
Nevertheless, this problem remains largely open.
Moreover, the known approaches to it generally do not
directly provide a description of the rational functions
$X$ and $Y$ parametrizing these components.

Let us also emphasize that, unlike the problem of determining
whether, for given $A$ and $B$, the curve \eqref{cur} has
a component of genus zero, deciding whether given rational functions $X$ and $Y$
have a common composite has no
obvious algorithmic solution.
The main difficulty is that there is no general upper
bound for the  degree of a generator of $\C(X)\cap \C(Y)$  in terms of the degrees of $X$ and $Y$.

Regardless of the genus of $\f E$, the intersection
problem can be easily solved if both $X$ and $Y$
are Galois coverings. Specifically, denoting by $G_X$ the finite group of
deck transformations of a map $X$, it is easy to see
that  two Galois coverings $X$ and $Y$ have a common
composite if and only if $\langle G_X,G_Y\rangle$ is finite. 
Moreover, in this case the generator $H$ of
$X^*M(\f R)\cap Y^*M(\f T)$ is a Galois covering
with $G_H=\langle G_X,G_Y\rangle$.  
In particular, if $g(\f E)\geq2$, then the automorphism
group of $\f E$ is finite by Hurwitz's theorem,
so $X$ and $Y$ always have a common composite.  

On the other hand, if $g(\f E)=0$, one can easily construct examples $X$ and $Y$ for which $\langle G_X, G_Y \rangle$ is infinite. For example, 
for $$X=z^2, \qquad Y=(z-1)^2,$$ we have
$$
G_X=\langle -z\rangle,\qquad G_Y=\langle 2-z\rangle,
$$
and 
\[
(2-z)\circ(-z)=z+2.
\] 
Thus, $X$ and $Y$ have no common composite.

Note that the fact that \eqref{us21} holds for $X$ and $Y$
given by \eqref{che} can be explained by the same principle.
Here $X$ and $Y$ need not be Galois coverings, but there
exists a rational function $W$ such that
\[
X_0=X\circ W\qquad {\rm and}\qquad Y_0=Y\circ W
\]
are Galois coverings and the group
$\langle G_{X_0},G_{Y_0}\rangle$ is finite.
Indeed, taking
\[
W=\frac{z+z^{-1}}{2},
\]
we obtain the Galois coverings
\[
X_0=T_n\circ W=\frac{z^n+z^{-n}}{2}
\qquad {\rm and} \qquad 
Y_0=T_m\circ W=\frac{z^m+z^{-m}}{2},
\]
whose deck transformation groups generate a finite
group---the dihedral group of order $2nm$.

Note also that using dihedral groups one can construct
examples of rational functions $X$ and $Y$ of degree
two such that the degree of the generator $H$ of
the intersection $\C(X)\cap\C(Y)$ is arbitrarily large.
Indeed, the deck transformation groups
of the rational functions
\[
X=z+\frac1z\qquad {\rm and} \qquad Y=z+\frac{\varepsilon}{z},
\]
where $\varepsilon$ is a primitive $n$th root of unity, 
$n\geq2$, generate a dihedral group of order $2n$.
Thus, \eqref{us2} does not generally
imply \eqref{us3} in the non-polynomial case.

In this paper, we solve the maximal intersection problem
in the case where {\it one} of the holomorphic maps
$X$ and $Y$ is a Galois covering.
It is clear that certain solutions arise when $X$ and $Y$
are both Galois coverings.
However, this class does not exhaust all possible solutions.

Indeed, let
$Y:\f E\longrightarrow\f S$
be a holomorphic map between compact Riemann surfaces,
and let $G\leq\Aut(\f E)$ be a finite subgroup of
automorphisms of $\f E$.
We say that $Y$ is {\it $G$-equivariant} if there exists
an injective homomorphism
\[
\phi_Y:G\longrightarrow\Aut(\f S)
\]
such that
\be \l{homo}
Y\circ\mu=\phi_Y(\mu)\circ Y,\quad \mu\in G.
\ee
Assume now that $X:\f E\rightarrow\f R$ is a Galois covering,
$Y:\f E\rightarrow\f T$ is $G_X$-equivariant,
and $B:\f T\rightarrow\f C$ is the quotient map
corresponding to $\phi_Y(G_X)$. Then $Y$ maps every
$X$-fiber to a $B$-fiber, implying that there exists
a holomorphic map $A:\f R\rightarrow\f C$ which makes
diagram~\eqref{sas} commute.
Moreover, one can show that \eqref{homo} implies
\eqref{us1}. Together with $\deg B=\deg X$, this
implies \eqref{us3}.

The above two mechanisms for producing solutions
to the maximal intersection problem can be combined
in the diagram
\begin{equation}
\begin{CD}
\f E @>V>> \f S @>U>> \f T\\
@VV X V @VV W V @VV B V\\
\f R @>C>> \f F @>D>> \f C,
\end{CD}
\end{equation}
where $V$ is $G_X$-equivariant,  $W$ is the quotient map corresponding to
$\phi_V(G_X)$, and  $U$
is a Galois covering such that the group
$\langle\phi_V(G_X),G_U\rangle$ is finite of order
$\deg X\cdot\deg U$.

The following result states that all solutions to the
maximal intersection problem in the case where $X$
is a Galois covering are obtained by this construction.
\begin{theorem}\l{t4} Let 
$
X: \f E \rightarrow \f R$ and $Y: \f E \rightarrow \f T
$
be 
holomorphic maps  between compact Riemann surfaces of degree at least two 
satisfying 
\[
X^* M(\f R) \cdot Y^* M(\f T) = M(\f E). 
\]
Assume that $X$ is a Galois covering.
Then the equality
\[
[ M(\f E) : X^* M(\f R) \cap Y^* M(\f T)] = \deg X \cdot \deg Y
\]
holds if and only if $Y$ admits a factorization
$
Y = U \circ V,
$
as a composition of holomorphic maps  between compact Riemann surfaces 
$
V: \f E \rightarrow \f S$ and $  U: \f S \rightarrow \f T
$, \linebreak
where $V$ is $G_X$-equivariant and $U$ is a Galois covering such that the group \linebreak
$
\langle \phi_V(G_X), G_U \rangle
$
is finite of order $\deg X \cdot \deg U$.

\end{theorem}

It is easy to see that if $\f E=\C\P^1$ and  $X=z^n$, then the $G_X$-equivariance property
means that 
\be \l{form} 
Y=\sigma\circ z^s R(z^n),
\ee 
where $\sigma$ is a M\"obius transformation, $R$ is
a rational function, and $s\geq0$ is an integer satisfying
$\gcd(s,n)=1$. This generalizes \eqref{fde0}, where $R$
is assumed to be a polynomial. 
Thus, the fact that \eqref{us21} holds for $X$ and $Y$
given by \eqref{fde0} is a particular case of the
above general principle based on equivariance.

For other rational Galois coverings $X$,
$G_X$-equivariance is harder to describe
but is relatively well understood. 
Specifically, for any finite subgroup
$G\subset\mathrm{Aut}(\mathbb C\mathbb P^1)$, the set
of rational functions $Y$ satisfying
\be \l{sat}
Y\circ\mu=\mu\circ Y,\qquad \mu\in G,
\ee
is nonempty and admits a rather explicit description
given in~\cite{dm}.
Our notion of equivariance is broader. However, it is
easy to see that if $Y$ is $G$-equivariant in the above
sense, then for some M\"obius transformation $\mu$
and some integer $l\geq1$, the iterate
$(\mu\circ Y)^{\circ l}$ satisfies \eqref{sat}. For general compact Riemann surfaces, we are not aware
of a similarly explicit description of equivariant
holomorphic maps.

Suppose that $X$ is a Galois covering and that $Y$
is indecomposable, that is, it cannot be represented
as a composition of two holomorphic maps of degree
at least two.
Then, by Theorem~\ref{t4}, conditions \eqref{us1}
and \eqref{us3} imply that $Y$ must be either
a Galois covering or $G_X$-equivariant.

This conclusion need not hold when condition~\eqref{us3}
is weakened to~\eqref{us2} (see Section~\ref{ex}).
Nevertheless, using local methods, we prove the following
result on equivariance under the weaker assumption
\eqref{us2}.
Note that $X$ here is not necessarily a Galois covering.

\bt \l{t22}
Let $X:\f E\to\f R$ and $Y:\f E\to\f T$ be holomorphic
maps of degree at least two such that $Y$ is indecomposable and 
\[
X^*M(\f R)\cdot Y^*M(\f T)=M(\f E),\qquad
X^*M(\f R)\cap Y^*M(\f T)\ne\mathbb C.
\]
Let $G$ be the subgroup of $G_X$ generated by all
elements $\mu\in G_X$ that fix at least one critical
point of $Y$.
Then $Y$ is $G$-equivariant.
\et

In particular, if $X=z^n$, $n\geq2$, and $Y$ has
a critical point at $0$ or $\infty$, then \linebreak $G=G_X$.
Under the other assumptions of the theorem, $Y$ is
therefore $G_X$-equivariant and has the form \eqref{form}. This again generalizes \eqref{fde0}.

Finally, we prove the following simple criterion
ensuring that holomorphic maps satisfying \eqref{us1}
have no common composite.   
For a holomorphic map $R:\f I\rightarrow\f J$
and a point $z_0\in\f I$, we denote by $\mult_{z_0}R$
the multiplicity of $R$ at $z_0$.

\begin{theorem}\l{t5}
Let 
$
X: \f E \rightarrow \f R$ and $Y: \f E \rightarrow \f T
$
be 
holomorphic maps  between compact Riemann surfaces of degree at least two such that
\[
X^* M(\f R) \cdot Y^* M(\f T) = M(\f E).
\]
Assume that there exists a point $z_0 \in \f E$ such that
\[
\gcd(\mult_{z_0} X, \mult_{z_0} Y) > 1.
\]
Then
\[
X^* M(\f R) \cap Y^* M(\f T) = \mathbb{C}.
\]
\end{theorem}

Note that  Theorem~\ref{t5} implies in particular  that polynomials $X$ and $Y$
satisfying \eqref{us11} and having non-coprime degrees
have no common composite---a conclusion that also
follows from the description of polynomial solutions
to~\eqref{h} given above.

This paper is organized as follows. In Section~2, we
recall the necessary definitions and results concerning
fiber products and normalizations of holomorphic maps
between compact Riemann surfaces. We then obtain several results on commutative diagrams
of the form \eqref{sas} involving Galois coverings
(Theorems~\ref{phi}--\ref{t6}) and deduce Theorem~\ref{t4}. 
The remaining part of the section is devoted to
Theorems~\ref{t22} and~\ref{t5}.
The former is proved by local methods, adapting
certain techniques from~\cite{form} on functional
equations in formal power series.
For the latter, we give two proofs: one based on
Abhyankar's lemma and the other using the same
local methods.

In Section~3, we discuss the rational case.
In particular, when $X$ is a rational Galois covering,
we give a description of solutions to the maximal
intersection problem in terms of semiconjugate
rational functions and Galois coverings
(Theorem~\ref{t7}).
We also describe commutative diagrams of rational functions
of the form \eqref{sas} with $B$ a Galois covering
(Theorem~\ref{t8}). 
Finally, we construct examples of rational functions
$X$ and $Y$ with $\deg X=2$ satisfying \eqref{us11}
and \eqref{us21} but not \eqref{us3}, for which $Y$
is indecomposable but is neither a Galois covering
nor $G_X$-equivariant.

\section{The intersection problem for holomorphic maps}
\subsection{\l{fib} Fiber products and normalizations} 
In this paper, all Riemann surfaces under consideration
are assumed to be compact, and all holomorphic maps
are assumed to be non-constant. 
For brevity, we shall usually use the expression ``holomorphic map''  to mean a holomorphic map between compact Riemann surfaces.

Let $\f R$, $\f T$, and $\f C$ be compact Riemann surfaces, and $A:\, \f R\rightarrow \f C$ and $B:\, \f T \rightarrow \f C$ be holomorphic maps. 
The collection
\be \l{nota} (\f R,A)\times_{\f C} (\f T,B)=\bigl\{\{\f E_j,X_j,Y_j\}:1\leq j\leq n(A,B)\bigr\},\ee 
where $n(A,B)$ is a positive integer and $\f E_j,$ $1\leq j \leq n(A,B),$ are compact Riemann surfaces provided with holomorphic maps
$$X_j:\, \f E_j\rightarrow \f R, \ \ \ Y_j:\, \f E_j\rightarrow \f T, \ \ \ 1\leq j \leq n(A,B),$$
is called the {\it fiber product} of  $A$ and $B$ if $$ A\circ X_j=B\circ Y_j, \ \ \ 1\leq j \leq n(A,B),$$ 
and for any holomorphic maps $X:\, \t{\f E}\rightarrow \f R,$  $Y:\, \t{\f E}\rightarrow \f T$
 satisfying 
\be \l{pes}  A\circ X=B\circ Y\ee there exist a unique index $j$,
$1\leq j\leq n(A,B)$, 
and a unique holomorphic map
$T:\,\t{\f E}\rightarrow\f E_j$ such that the equalities 
\be \l{ae} X= X_j\circ  T, \ \ \ Y= Y_j\circ T\ee  
hold. 

The fiber product  is defined in a unique way up to natural isomorphisms and 
can be described by the following algebro-geometric construction. Let us consider the algebraic variety 
\be \l{ccuurr} \f L=\{(x,y)\in \f R\times \f T \, \vert \,  A(x)=B(y)\}.\ee
Let us denote by $\f L_j,$ $1\leq j \leq n(A,B)$,  the irreducible components of $\f L$, by 
$\f E_j$, \linebreak  $1\leq j \leq n(A,B)$, their desingularizations, 
 and by $$\pi_j: \f E_j\rightarrow \f L_j, \ \ \ 1\leq j \leq n(A,B),$$ the desingularization maps.
Then the compositions  $$x\circ \pi_j: \f E_j\rightarrow \f R, \ \ \ y\circ \pi_j: \f E_j\rightarrow \f T, \ \ \ 1\leq j \leq n(A,B),$$ 
where $x$ and $y$ denote the projections of $\f R\times\f T$
onto its factors,
define holomorphic maps
$$X_j:\, \f E_j\rightarrow \f R, \ \ \ Y_j:\, \f E_j\rightarrow \f T, \ \ \ 1\leq j \leq n(A,B),$$
and the collection $$\bigl\{\{\f E_j,X_j,Y_j\}:1\leq j\leq n(A,B)\bigr\}$$ is the fiber product of $A$ and $B$. 
Abusing notation, we call the Riemann  surfaces $\f E_j$,  $1\leq j \leq n(A,B),$ irreducible components of the fiber product of $A$ and $B$.

It follows from the definition  that for every $j,$ $1\leq j \leq n(A,B),$ the maps $X_j,Y_j$  have no {\it non-trivial common compositional  right factor} in the following sense: 
the equalities 
$$ X_j= \widehat X\circ  T, \ \ \ Y_j= \widehat Y\circ T,$$ where $$T:\, {\f E}_j \rightarrow \widehat{\f E} , \ \ \ \widehat X:\, \widehat{\f E} \rightarrow \f R, \ \ \ \widehat Y:\,  \widehat{\f E}  \rightarrow \f T$$ are holomorphic maps, imply that $\deg T=1.$  
We can restate this condition as the equality 
$$ X_j^* M(\f R)\cdot Y_j^* M(\f T)= M(\f E_j),$$ meaning that the field $ M(\f E_j)$ is the compositum  of its subfields $X_j^* M(\f R)$ and $Y_j^* M(\f T).$  
In the other direction, if $X$ and $Y$ satisfy \eqref{pes} and have no non-trivial common compositional  right factor, then   
 $$X=X_j\circ T, \ \ \ \ Y=Y_j\circ T$$ for some $j$, $1\leq j \leq n(A,B),$ and an  isomorphism $T:\, \t{\f E}\rightarrow \f E_j.$  

Notice that since $X_j,Y_j$, $1\leq j \leq n(A,B),$  parametrize components of   
\eqref{ccuurr}, the equalities    
\be \l{ii} \sum_{j}\deg X_j=  \deg B,\ \ \ \ \sum_{j}\deg Y_j= \deg A\ee
hold. In particular, if $(\f R,A)\times_{\f C} (\f T,B)$ consists of a single component $\{\f E,X,Y\},$ then \be \l{vv} \deg X=\deg B,\ \ \ \ \deg Y=\deg A.\ee Conversely,  if holomorphic maps $Y$ and $X$ satisfy \eqref{pes}  and \eqref{vv}, and have no non-trivial common compositional  right factor,  then $(\f R,A)\times_{\f C} (\f T,B)$ consists of a single component.

Let $X:\, \f E\rightarrow \f R$ be a holomorphic map. 
Let us recall that $X$ is called a \textit{Galois covering} if its deck transformation group
\[
\operatorname{Aut}(\f E,X)=\{\mu \in \operatorname{Aut}(\f E) \mid X\circ\mu = X\},
\]
which we denote by $G_X$, acts transitively on the fibers of $X$. 
Equivalently, $X$ is a Galois covering if the field extension 
$ M(\f  E)/X^* M(\f R)$ is a Galois extension. 
If $X$ is a Galois covering, we have the isomorphism
$$\Aut(\f E,X)\cong{\rm Gal}\left(M(\f E)/X^*M(\f R)\right).$$
Note that a holomorphic map $X$ is a Galois covering
if and only if
\be \l{dega}
|G_X|=\deg X.
\ee 

If $X:\f E\to\f R$ is a Galois covering, then for any subgroup $G\subseteq G_X$, the quotient map
$
T:\f E\longrightarrow\f E/G
$
is a Galois covering with $G_T=G$.
Since $X$ is constant on the $G$-orbits, there exists
a holomorphic map $\widehat X:\f E/G\to\f R$ such that
\be \l{decc}
X=\widehat X\circ T.
\ee
Conversely, in every decomposition \eqref{decc} into
holomorphic maps, $T$ is also a Galois covering and
$G_T$ is a subgroup of $G_X$.

Let $X:\, \f E\rightarrow \f R$  be  an arbitrary holomorphic map.  Then the {\it normalization} of $X$ is defined as a compact Riemann surface  $\f N_X$ together with a Galois covering  of the 
lowest possible degree $\t X:\f N_X\rightarrow \f R$ such that
 $\t X=X\circ H$ for some  holomorphic map $H:\,\f N_X\rightarrow \f E$. 
The map $\t X$ is defined up to the change $\t X\rightarrow \t X \circ \alpha,$ where $\alpha\in\Aut(\f N_X)$, and is 
characterized by the property that  the field extension 
$ M(\f N_X)/{\t X}^* M(\f R)$ is isomorphic to the Galois closure $\t{ M(\f E)}/X^* M(\f R)$
of the extension $ M(\f E)/X^* M(\f R)$.

\subsection{\l{2.2} Good solutions of $A \circ X = B \circ Y$} 
When studying the functional equation
\begin{equation}\label{fe}
A \circ X = B \circ Y
\end{equation}
in holomorphic maps, 
which we will usually write in the form of the commutative diagram
\be \l{bell} \begin{CD}
\f E @>Y>> \f T \\
@VV X V @VV B V \\
\f R @>A>> \f C\,,
\end{CD}
\ee
it is convenient to use the following concept.   
Let $X, Y, A, B$ be holomorphic maps satisfying \eqref{fe}. 
We say that $X, Y, A, B$ form a \emph{good solution} of~\eqref{fe} if 
the fiber product of $A$ and $B$ has a unique component, and $X$ and $Y$ have no nontrivial common compositional right factor.

For a good solution $X, Y, A, B$, the equalities~\eqref{vv} hold. Furthermore, the above properties of fiber products easily imply the following statement, which provides a criterion for a solution to be good.

\begin{lemma} \l{21}
A solution $X, Y, A, B$ of the equation
$A\circ X=B\circ Y$ is good whenever any two of the following
three conditions are satisfied:
\begin{itemize}
    \item The fiber product of $A$ and $B$ has a unique component,
    \item The maps $X$ and $Y$ have no non-trivial common compositional right factor,
    \item The equalities $\deg B = \deg X$, $\deg A = \deg Y$ hold. \qed
\end{itemize}
\end{lemma}

The problem considered in this paper is related to the concept of a good solution in the following way.

\bl \l{eqi} 
Let $X: \f E \rightarrow \f R$ and $Y: \f E \rightarrow \f T$ be holomorphic maps  between compact Riemann surfaces  of degree at least two  such that 
\be \label{u1}
X^* M(\f R) \cdot Y^* M(\f T) = M(\f E).
\ee
 Then the condition
\be \label{u2}
[ M(\f E) : X^* M(\f R) \cap Y^* M(\f T)] = \deg X \cdot \deg Y
\ee
holds if and only if there exist holomorphic maps $A$ and $B$ between compact Riemann surfaces such that $X$, $Y$, $A$,  $B$ form a good solution of the equation $A \circ X = B \circ Y.$
\el
\pr Suppose that $X$ and $Y$ have a common composite. 
Then there exist a compact Riemann surface $\f C$
and holomorphic maps $H:\f E\to\f C$ and $A:\f R\to\f C$, $B:\f T\to\f C$ such that
\be \l{ind2}
X^*M(\f R)\cap Y^*M(\f T)=H^*M(\f C) 
\ee
and 
\be \l{ind}
H=A\circ X=B\circ Y.
\ee

By \eqref{u1}, $\{\f E,X,Y\}$ is a component of
the fiber product of $A$ and $B$.
Therefore, \eqref{ii} gives $\deg A\geq\deg Y$, and hence
\be \l{bou} 
[M(\f E):X^*M(\f R)\cap Y^*M(\f T)]
=\deg H=\deg X\cdot\deg A
\geq\deg X\cdot\deg Y.
\ee 
If \eqref{u2} holds, then equality holds here, so
$\deg A=\deg Y$.
Equality \eqref{fe} then also gives
$\deg B=\deg X$. Thus, $X,Y,A,B$ form a good solution  
by Lemma~\ref{21}.

Conversely, suppose that $X,Y,A,B$ form a good solution. Then for the map $H$ defined by equality \eqref{ind} we have 
\[
\deg H=\deg X\cdot\deg Y,
\]
by \eqref{vv}.
Moreover,
\[
H^*M(\f C)\subseteq X^*M(\f R)\cap Y^*M(\f T).
\]
Therefore,
\[
[M(\f E):X^*M(\f R)\cap Y^*M(\f T)]
\leq\deg H=\deg X\cdot\deg Y.
\]
Combining this with the lower bound \eqref{bou} 
gives \eqref{u2}. \qed

The main technical statement we need is the following ``lifting lemma'' 
(see \cite{plo}, Lemma~3.4).

\bl \l{lift}
 Let $X, Y, A, B$ be holomorphic maps  between compact Riemann surfaces  of degree at least two  that form a good solution of the equation $A \circ X = B \circ Y$. Then the diagram 
\be \l{d1} 
\begin{CD}
\f E @>Y>> \f T\\
@VV X V @VV B V \\
\f R @>A >> \f C
\end{CD}
\ee
 can be completed  to a diagram of holomorphic maps  between compact Riemann surfaces 
\be \l{d2} 
\begin{CD}
\f N_X @>L >> \f N_B\\
@VV Q V @VV F V \\
\f E @>Y>> \f T\\
@VV X V @VV B V \\
\f R @>A >> \f C,
\end{CD}
\ee
such that $X\circ Q=\t X$    
 and $B\circ F= \t B$. \qed
\el

\subsection{Solutions of $A\circ X=B\circ Y$ involving Galois coverings}

Let $X:\, \f E \rightarrow \f R$, $Y:\, \f E \rightarrow \f T$, and $B:\, \f T \rightarrow \f C$ be holomorphic maps. 
Then a holomorphic map $A:\, \f R \rightarrow \f C$ making the diagram \eqref{bell} commutative exists if and only if $Y$ maps every $X$-fiber to a $B$-fiber.  
If both $X$ and $B$ are Galois covers, an obvious sufficient condition for this is the existence of a homomorphism
\be \l{ho1} 
\phi : G_X \longrightarrow G_B
\ee 
such that
\be \l{ho2}
Y \circ \mu = \phi(\mu) \circ Y, 
\quad  \mu \in G_X.
\ee 

The following result shows that this condition is also necessary and relates good solutions  to equivariant maps. Related results were obtained in Theorem~5.1
of~\cite{semi} and Theorem~3.2 of~\cite{hur}.

\bt \l{phi} Let $X, Y, A, B$ be holomorphic maps between compact Riemann surfaces such that
the diagram
$$
\begin{CD}
\f E @>Y>> \f T \\
@VV X V @VV B V \\
\f R @>A>> \f C
\end{CD}
$$ 
commutes. 
Assume that $X$ and $B$ are Galois coverings. 
Then there exists a homomorphism
\[
\phi : G_X \longrightarrow G_B
\]
 such that
\[
Y \circ \mu = \phi(\mu) \circ Y, 
\quad  \mu \in G_X.
\]
Moreover, \[
X^*  M(\f R) \cdot Y^*  M(\f T) =  M(\f E)
\]
if and only if $\phi$ is injective, in which case
\[
n(A,B)=\frac{\deg B}{\deg X}.
\]
Finally, $\phi$ is an isomorphism
if and only if 
$X, Y, A, B$ form a good solution of the equation $A \circ X = B \circ Y$. 
\et
\pr 
Clearly, equality \eqref{fe} implies that 
for every $\mu\in G_X$ the equality  \be \l{as} B\circ (Y\circ \mu)=B\circ Y \ee holds. On the other hand, for the fiber product of $B$ with itself, the 
maps $X_j,Y_j$ in \eqref{nota} 
are 
$$X_j=\mu_j, \ \ \ Y_j=\mathrm{id}, \ \ \ \mu_j\in G_B.$$ Indeed, clearly,  $X_j,Y_j$ defined in this way  satisfy \eqref{pes} and have no non-trivial common compositional  right factor. Moreover, since 
$$\sum_j\deg \mu_j=\vert G_B\vert =\deg B,$$ by \eqref{dega}, these maps    exhaust all  $X_j,Y_j$ in \eqref{nota} by 
\eqref{ii}. Thus, the universal property
\be
Y\circ \mu = X_j\circ T, \quad Y = Y_j\circ T
\ee
for the solutions $Y$ and $Y\circ\mu$ of \eqref{as} reduces to the equality
\be \l{hhh}
Y\circ\mu = \phi(\mu)\circ Y
\ee
for some $\phi(\mu) \in G_B$. Finally, one can easily see that the correspondence $\mu \mapsto \phi(\mu)$ defines a homomorphism of groups.

To prove the second statement of the theorem, let us observe that, since $X$ is a Galois covering, if the equalities 
\be \l{ew}
X = \widehat{X} \circ T, \quad Y = \widehat{Y} \circ T,
\ee
hold 
for some holomorphic maps, then the first of these equalities implies that $T$ is also  a Galois covering and that $G_T$ is a subgroup of $G_X$. Furthermore, since 
$$T\circ \mu=T, \qquad \mu\in G_T,$$ 
the second equality implies that the kernel of $\phi$ contains $G_T$.
It follows readily that the equalities~\eqref{ew} hold for some holomorphic map $T$ of degree at least two if and only if the kernel of $\phi$ is non-trivial.

To prove the degree formula, assume that $\phi$ is
injective. Then for every $g\in G_B$ the maps
$X$ and $g\circ Y$ satisfy \eqref{fe} and have no
non-trivial common compositional right factor.
Thus, each pair parametrizes a component of the fiber
product of $A$ and $B$.  
Moreover, two pairs $(X,g\circ Y)$ and $(X,h\circ Y)$ parametrize the same component if and only if there exists
an automorphism $\mu$ of $\f E$ such that
\[
X\circ\mu=X,\qquad g\circ Y\circ\mu=h\circ Y.
\]
The first equality means that $\mu\in G_X$, and
the second yields $h=g\phi(\mu)$ by \eqref{hhh}. 
Thus, since $\phi$ is injective, for each $g$ there are
exactly $|G_X|$ such elements $h\in G_B$.

Selecting one representative pair $(X_j,Y_j)$ for each
component parametrized in this way, we therefore obtain 
\[
\sum_j\deg X_j
=\frac{|G_B|}{|G_X|}\deg X
=\deg B.
\]
Hence these pairs exhaust all components by \eqref{ii}, and 
\[
n(A,B)=\sum_j\frac{\deg X_j}{\deg X}
=\frac{\deg B}{\deg X}.
\]

Finally, if $\phi$ is injective, then Lemma~\ref{21}
implies that $X,Y,A,B$ form a good solution if and only if
$\deg X=\deg B$. For Galois coverings $X$ and $B$,
this equality is equivalent to $|G_X|=|G_B|$, which 
is equivalent to the surjectivity of $\phi$.
\qed

The next result concerns solutions of \eqref{fe}
in the case where $X$ and $Y$ have no non-trivial
common compositional right factor and $B$ is a Galois
covering.

\bt \l{bgal}
Let $X, Y, A, B$ be holomorphic maps between compact Riemann surfaces such that
the diagram
$$
\begin{CD}
\f E @>Y>> \f T \\
@VV X V @VV B V \\
\f R @>A>> \f C
\end{CD}
$$ 
commutes and  \[
X^*M(\f R)\cdot Y^*M(\f T)=M(\f E).
\]
Assume that $B$ is a Galois covering.
Then $X$ is also a Galois covering.
\et
\pr Consider the following diagram of fields:
\[
\begin{CD}
M(\f T) @>{Y^*}>> Y^*M(\f T) @>>> M(\f E)\\
@AAA @AAA @AAA\\
B^*M(\f C) @>{\overline{Y^{*}}}>> (B\circ Y)^*M(\f C)
@>>> X^*M(\f R)
\end{CD}
\]
where the upper arrow $Y^*$ is the pullback isomorphism
induced by $Y$, the lower arrow $\overline{Y^*}$ is its
restriction to $B^*M(\f C)$, and the unlabelled arrows
are inclusions.

Since $B$ is a Galois covering, the extension in the
left column is Galois, and hence so is the extension
in the middle column, since $Y^*$  is an isomorphism. Now 
Proposition~3.19 in~\cite{milneFT} implies that the
extension in the right column is also Galois.
Thus, $X$ is a Galois covering. \qed

The following simple result provides conditions for
two Galois coverings to have a common composite.

\bt \l{2.6} 
Let $X: \f E \rightarrow \f R$ and $Y: \f E \rightarrow \f T$ be holomorphic maps between compact Riemann surfaces of degree at least two.   
Assume that $X$ and $Y$ are Galois coverings. Then the condition
\be \l{be} 
X^* M(\f R) \cap Y^* M(\f T) = H^* M(\f C)
\ee
holds for some non-constant holomorphic map $H : \f E \rightarrow \f C$ 
if and only if the group $\langle G_X, G_Y \rangle$ is finite, in which case it holds for some Galois covering $H$ with \linebreak $G_H = \langle G_X, G_Y \rangle$.
Furthermore, the equality
\[
X^*M(\f R)\cdot Y^*M(\f T)=M(\f E)
\]
holds if and only if $G_X\cap G_Y=\{e\}$.
Finally, the equality
\be \l{wi}
[M(\f E):X^*M(\f R)\cap Y^*M(\f T)]
=\deg X\cdot\deg Y
\ee
holds if and only if
\be \l{fu}
|\langle G_X,G_Y\rangle|=\deg X\cdot\deg Y.
\ee
\et

\pr Let us set $G = \langle G_X, G_Y \rangle$. 
If there exist holomorphic maps $A$, $B$, and $H$ satisfying \eqref{ind}, then $G$ is clearly a subgroup of $\Aut(\f E, H)$ and therefore finite. Moreover, $\deg H \ge |G|$.  Conversely, if $G$ is finite, we can define the quotient map
\[
H:\f E\longrightarrow\f C=\f E/G.
\]
Since $G_X,G_Y\subseteq G$, the map $H$ factors through
both $X$ and $Y$, giving holomorphic maps $A$ and $B$
such that \eqref{ind} holds.
Moreover, the intersection in \eqref{be} consists exactly
of the meromorphic functions invariant under both $G_X$
and $G_Y$, or equivalently, under $G$.
Since $H$ is the quotient map corresponding to $G$,
this intersection equals $H^*M(\f C)$.
This proves the first part of the theorem.

Furthermore, it follows from the description of decompositions of Galois coverings that $X$ and $Y$ have no common nontrivial compositional right factor if and only if $G_X \cap G_Y = \{e\}$. Finally, if $G$ is finite, the quotient map $H$
constructed above satisfies
\[
[M(\f E):X^*M(\f R)\cap Y^*M(\f T)]
=\deg H=|G|.
\]
Thus, \eqref{wi} and \eqref{fu} are equivalent. \qed

The proof of Theorem~\ref{t4} is based on the following
statement, which is the main result of this section.

\bt \l{t6}
Let $X,Y,A,B$ be holomorphic maps between compact
Riemann surfaces of degree at least two.
Assume that $X$ is a Galois covering and \linebreak $\deg B=\deg X$.
Then $X, Y, A, B$ form a good solution of the equation $A \circ X = B \circ Y$
if and only if there exists a commutative diagram of holomorphic maps
\[
\begin{CD}
\f E @>V>> \f S @>U>> \f T \\
@VV X V @VV W V @VV B V \\
\f R @>C>> \f F @>D>> \f C
\end{CD}
\]
such that
$
Y = U \circ V,$ $ A = D \circ C,$ 
and the following conditions  are satisfied:

\begin{enumerate}  \renewcommand{\labelenumi}{\textup{(\arabic{enumi})}}
 \item The equalities $$\f S=\f N_B \quad \text{and} \quad\t B=B\circ U$$ hold. In particular, the maps $W$ 
and $U$ are Galois coverings.

    \item There exists an isomorphism
    $
    \phi \colon G_X \longrightarrow G_W
    $
    such that
    \begin{equation}
    V \circ \mu = \phi(\mu) \circ V, \qquad \mu \in G_X. 
    \end{equation}
    \item The groups $G_W$ and $G_U$ satisfy
$
G_W\cap G_U=\{e\}.
$

\end{enumerate}
\et 
\pr
If $X$, $Y$, $A$, $B$ form a good solution of \eqref{fe}, then by Lemma~\ref{lift}, taking into account that $\t X=X$ because $X$ is a Galois cover,
the map $Y$ can be decomposed into a composition of holomorphic maps as $Y=U\circ V$,
so that the diagram
\[
\begin{tikzcd}[row sep=2.75em, column sep=2.75em]
& \f N_B \arrow[d,"U"] \\
\f E \arrow[r,"Y"] \arrow[d,"X"'] \arrow[ur,"V"] & \f T \arrow[d,"B"] \\
\f R \arrow[r,"A"'] & \f C
\end{tikzcd}
\]
commutes and the equality $\t B=B\circ U$ holds.

Since the diagram
\[
\begin{CD}
\f E @>V>> \f N_B \\
@VV X V @VV \t B V \\
\f R @>A>> \f C
\end{CD}
\]
commutes, by Theorem~\ref{phi} there exists a homomorphism
\[
\phi:G_X\longrightarrow G_{\t B}
\]
such that
\[
V\circ\mu=\phi(\mu)\circ V,\qquad \mu\in G_X.
\]

Since $X$ and $Y$ have no non-trivial common compositional
right factor, the same is true of $X$ and $V$.
Thus, $\phi$ is injective by Theorem~\ref{phi}.
The image of this homomorphism is a subgroup of $G_{\t B}$
and therefore has the form $G_W$ for some Galois covering $W$
such that $\t B=D\circ W$ for a suitable holomorphic map $D$.
Furthermore, since $V$ maps every $X$-fiber to a $W$-fiber,
there exists a holomorphic map $C$ such that the diagram
\[
\begin{CD}
\f E @>V>> \f N_B \\
@VV X V @VV W V \\
\f R @>C>> \f F
\end{CD}
\]
commutes.

The commutativity of the last diagram combined with the equalities
$$\t B=D\circ W=B\circ U$$
implies that the diagram
\[
\begin{CD}
\f E @>V>> \f N_B @>U>> \f T \\
@VV X V @VV W V @VV B V \\
\f R @>C>> \f F @>D>> \f C
\end{CD}
\]
commutes, and since
$$D\circ C\circ X=B\circ U\circ V=B\circ Y=A\circ X,$$
the equality $A=D\circ C$ holds.
Moreover, since $U$ is a compositional right factor of the
Galois covering $\t B$, it is also a Galois covering.

To prove condition \textup{(3)}, assume that
$G_W\cap G_U\ne\{e\}$.
Then, since $\phi:G_X\to G_W$ is an isomorphism,  there exists $\mu\in G_X$, $\mu\ne e$, such that
$\phi(\mu)\in G_U$.
For this $\mu$, we have
\[
Y\circ\mu=U\circ\phi(\mu)\circ V=U\circ V=Y.
\]
Since also $X\circ\mu=X$, both $X$ and $Y$ factor
through the quotient map $\f E\to\f E/\langle\mu\rangle$,
whose degree is at least two.
This contradicts the fact that $X$ and $Y$ have no
non-trivial common compositional right factor.
Thus, $G_W\cap G_U=\{e\}$.

By Lemma~\ref{21}, to finish the proof, we only need to
show that if conditions \textup{(1)--(3)} hold, then
$X$ and $Y$ have no non-trivial common compositional
right factor.
Since $X$ is a Galois covering, any such factor $T$
is also a Galois covering, and $G_T$ is a subgroup of $G_X$. 
For $\mu\in G_T$, the equality $Y\circ\mu=Y$ implies
\[
U\circ\phi(\mu)\circ V=U\circ V,
\]
and therefore $\phi(\mu)\in G_W\cap G_U=\{e\}$. Thus, $G_T=\{e\}$ by the injectivity of $\phi$,
and hence $\deg T=1$. \qed 

\vskip 0.2cm
\noindent{\it Proof of Theorem \ref{t4}.}
The ``only if'' part follows from Lemma~\ref{eqi}
and Theorem~\ref{t6}.
Indeed, in the diagram of Theorem~\ref{t6}, we have
$$\deg W=\deg X=\deg B.$$
Since $G_W\cap G_U=\{e\}$ and both groups are subgroups
of $G_{\t B}$, we obtain 
\[
\deg X\cdot\deg U
=|G_W|\,|G_U|
\leq|\langle G_W,G_U\rangle|
\leq|G_{\t B}|=\deg X\cdot\deg U.
\]
Thus,
\be \l{mor}
|\langle \phi_V(G_X),G_U\rangle|
=\deg X\cdot\deg U.
\ee

To prove the ``if'' part, suppose that $Y=U\circ V$
is a factorization with the stated properties. 
By \eqref{us1} and Lemmas~\ref{21} and~\ref{eqi},
it suffices to construct holomorphic maps $A$ and $B$
such that \eqref{fe} holds and $\deg B=\deg X$.

Set $G=\langle\phi_V(G_X),G_U\rangle$, and let
$Q:\f S\rightarrow\f C$ and $W:\f S\rightarrow\f F$
be the quotient maps corresponding to $G$ and
$\phi_V(G_X)$.
Then
\be \l{epta} 
Q=D\circ W=B\circ U
\ee
for suitable holomorphic maps
$D:\f F\rightarrow\f C$ and $B:\f T\rightarrow\f C$.
Since $V$ is $G_X$-equivariant, there exists a holomorphic
map $C:\f R\rightarrow\f F$ such that
\[
W\circ V=C\circ X.
\]
Setting $A=D\circ C$, we obtain \eqref{fe}.
Moreover, \eqref{epta} and \eqref{mor} give
\[
\deg B=\frac{|G|}{\deg U}=\deg X. \eqno{\Box}
\]

As mentioned in the introduction, for any finite subgroup
$G\subseteq\Aut(\C\P^1)$, the set of rational functions
commuting with every element of $G$ is nonempty and has
been described in~\cite{dm}. For a general compact Riemann
surface $\f E$ and a finite subgroup $G\subseteq\Aut(\f E)$,
we do not know a general criterion for the existence of
$G$-equivariant maps of degree at least two.

When $G\cong\Z/2\Z$, however, such maps exist for every
pair $(\f E,G)$ and admit a simple explicit construction.
Indeed, write $G=\langle\iota\rangle$, choose
$h\in M(\f E)$ with $h\circ\iota\ne h$, and set
\[
    f=h-h\circ\iota.
\]
Then $f$ is nonconstant and $f\circ\iota=-f$. Thus, $f$ is $G$-equivariant with
respect to the target involution $w\mapsto-w$.
Moreover, for every odd integer $k\ge3$, the map
\[
    Y=f^k:\f E\longrightarrow\C\P^1
\]
 is also $G$-equivariant  and has degree $k\,\deg f$.

Another class of examples is provided by maps between
hyperelliptic Riemann surfaces of genus at least two:
every nonconstant holomorphic map between such surfaces
is equivariant with respect to their hyperelliptic
involutions~\cite{mar,tan}.

\subsection{Holomorphic maps sharing a critical point}
In this section, we adapt the approach developed in~\cite{form}
for studying functional equations in formal power series
to the setting of holomorphic maps between compact Riemann surfaces.
Although the two settings differ considerably, some of the
techniques from~\cite{form} still apply in the present context.

The idea of our approach is as follows. Let
$X:\f E\to\f R$ and $Y:\f E\to\f T$ be holomorphic maps
having a common composite
\be \l{e1}
H = A \circ X = B \circ Y,
\ee
and suppose that $X$ and $Y$ share a critical point $z_0\in\f E$.
Choose local coordinates centered at $z_0$, $X(z_0)$,
$Y(z_0)$, and $H(z_0)$.
In local calculations below, we use the same letters
for the maps and for the germs of their expressions
in the chosen coordinates. These germs satisfy
\be \l{ee2}
X(0)=Y(0)=A(0)=B(0)=H(0)=0.
\ee
Moreover, the germs $X$, $Y$, and $H$ have zeros of
multiplicity at least two.

Let us denote by $\mathcal{O}_0$ the local ring of germs of analytic functions at $0$,
and by $\mathfrak{m}_0^k$ the ideal of $\mathcal{O}_0$ consisting of functions
having a zero of multiplicity at least $k$ at $0$. Finally, let $\mathcal G_0 = \mathfrak{m}_0^1 \setminus \mathfrak{m}_0^2$ be the group of germs of analytic functions at $0$ that vanish at $0$ and have a non-zero derivative, with group operation given by composition.

For a nonzero germ $f\in\mathfrak m_0^2$ of an analytic
function at zero, let us define a {\it transition function} 
$\phi_f$ as a germ $\phi_f\in\mathcal G_0$ such that
\begin{equation}\label{c}
f\circ\phi_f=f.
\end{equation}
It is clear that such germs form a group, which we denote
by $\Gamma_f$.
This group can be described in terms of a B\"ottcher
function (see \cite[Lemma~3.1]{form}) and plays an important
role in the study of functional equations in formal
power series (see \cite[Theorems~1.2 and~1.3]{form}).
Below we will need only the following simple result.

\bl \l{b1}
Let $f\in\mathfrak m_0^2$ have a zero of multiplicity
$n$ at zero. Then $\Gamma_f$ is cyclic of order $n$.
\el
\pr Since $f$ has a zero of multiplicity $n$, we can write
\[
f=z^n\circ u,\qquad u\in\mathcal G_0.
\]
Then $f\circ\phi=f$ if and only if
$u\circ\phi=\varepsilon u$ for some $\varepsilon\in U_n$,
where $U_n$ denotes the group of $n$th roots of unity.
Consequently,
\[
\Gamma_f=
\{u^{-1}\circ(\varepsilon z)\circ u:\varepsilon\in U_n\},
\]
which is cyclic of order $n$. \qed

The next lemma is a version of Corollary~4.4 in~\cite{form}.

\bl \l{tt2}
Let $X:\f E\to\f R$ and $Y:\f E\to\f T$ be holomorphic maps.
Assume that there exists a non-identity germ $\mu$ of a local holomorphic self-map
of $\f E$ such that
\[
X\circ\mu=X,\qquad Y\circ\mu=Y.
\]
Then
\be \l{e2} 
X^*M(\f R)\cdot Y^*M(\f T)\ne M(\f E).
\ee 
\el
\begin{proof}
If $X^*M(\f R)\cdot Y^*M(\f T)=M(\f E)$, then the
equalities in the hypothesis imply that
$f\circ\mu=f$ for every $f\in M(\f E)$.
Since meromorphic functions on a compact Riemann surface
separate points, we conclude that $\mu$ is the identity germ,
contrary to the assumption.
\end{proof}

\noindent{\it Proof of Theorem~\ref{t22}.}
Let $H$ be a common composite of $X$ and $Y$. 
Fix a critical point $a$ of $Y$ and an element
$\mu\in G_X$ such that $\mu(a)=a$. 
Choose local coordinates centered at $a$, $X(a)$,
$Y(a)$, and $H(a)$.
In these coordinates, the germs of $X$, $Y$, and $H$ vanish at zero.

Choose a non-identity element $\tau\in\Gamma_Y$,
which exists because $a$ is a critical point of $Y$.
Since $H=B\circ Y$, we have $\tau\in\Gamma_H$.
Moreover, since $H=A\circ X$ and $X\circ\mu=X$,
the germ of $\mu$ at $a$, expressed in the chosen local
coordinate, also belongs to $\Gamma_H$.
By Lemma~\ref{b1}, these germs commute. Hence, as germs at zero,
\[
Y\circ\mu\circ\tau
=Y\circ\tau\circ\mu
=Y\circ\mu.
\]
Thus, the germ $\tau$ at $a\in\f E$ satisfies
 the hypotheses of Lemma~\ref{tt2} for $Y$ and
$Y\circ\mu$. 

Hence,
\[
Y^*M(\f T)\cdot (Y\circ\mu)^*M(\f T)\ne M(\f E)
\]
and therefore $Y$ and $Y\circ\mu$ have a common
compositional right factor $W$ of degree at least two:
\[
Y=Y_1\circ W,\qquad Y\circ\mu=Y_2\circ W.
\]
By the indecomposability of $Y$, we have $\deg Y_1=1$.
Since $Y$ and $Y\circ\mu$ have the same degree,
we also have $\deg Y_2=1$. It follows that
\[
Y\circ\mu=\phi(\mu)\circ Y,
\qquad \phi(\mu)=Y_2\circ Y_1^{-1}\in\Aut(\f T).
\]

Since $G$ is generated by such elements $\mu$,  composing the
corresponding relations gives 
\[
Y\circ\mu=\phi(\mu)\circ Y,\qquad \mu\in G,
\]
where $\phi(\mu)\in\Aut(\f T)$.
Since $\phi(\mu)$ is uniquely determined by $\mu$,
the correspondence $\mu\mapsto\phi(\mu)$ defines
a homomorphism
\[
\phi:G\longrightarrow\Aut(\f T).
\]
Moreover, if $\phi(\mu)=e$, then
\[
X\circ\mu=X,\qquad Y\circ\mu=Y,
\]
implying that $\mu=e$ by Lemma~\ref{tt2} (or as in the proof of Theorem~\ref{t6}). Thus,  $\phi$ is injective and hence  $Y$ is $G$-equivariant. \qed

\vskip 0.2cm
\noindent{\it Proof of Theorem~\ref{t5}.}
Assume that there exist holomorphic maps $A$, $B$, and $H$
such that the equalities \eqref{ind2} and \eqref{ind} hold.
Then the assumption
\[
X^*M(\f R)\cdot Y^*M(\f T)=M(\f E)
\]
implies that $\{\f E,X,Y\}$ is a component of the fiber
product of $A$ and $B$. Therefore, by Abhyankar's lemma
(see, e.g.,~\cite[Theorem~3.9.1]{sti}), the equality
\be \l{abj}
\mult_{t_0}H
=\operatorname{lcm}\bigl(\mult_{X(t_0)}A,\;
\mult_{Y(t_0)}B\bigr)
\ee
holds for every point $t_0$ of $\f E$.
In turn, this equality implies that the numbers
\[
\mult_{t_0}X
=\frac{\mult_{t_0}H}{\mult_{X(t_0)}A}
\qquad\text{and}\qquad
\mult_{t_0}Y
=\frac{\mult_{t_0}H}{\mult_{Y(t_0)}B}
\]
are coprime, contrary to the hypothesis. 

The theorem can also be proved using the local approach. Indeed, assume that \eqref{e1} holds and choose local coordinates centered at $z_0$, $X(z_0)$,
$Y(z_0)$, and $H(z_0)$.
Then $\Gamma_X$ and $\Gamma_Y$ are subgroups of
$\Gamma_H$.
By Lemma~\ref{b1}, the group $\Gamma_H$ is cyclic, and hence
\[
|\Gamma_X\cap\Gamma_Y|
=\gcd(\mult_{z_0}X,\mult_{z_0}Y)>1.
\]
Thus, $X$ and $Y$ have a common non-identity transition
function at $z_0$, and Lemma~\ref{tt2} gives \eqref{e2},
contrary to the assumption. \qed

\section{The intersection problem for rational functions}

\subsection{Solutions of $A\circ X = B\circ Y$ involving a Galois covering}
In the case of rational functions, Theorem~\ref{t6}
takes the following more explicit form, linking the
maximal intersection problem to the semiconjugacy
relation between rational functions.
This relation has been intensively studied in recent
years (see \cite{semi}, \cite{rec}, and \cite{lattes}
for a comprehensive treatment).

\bt \l{t7}
Let $X,Y,A,B$ be rational functions of degree at least two.
Assume that $X$ is a Galois covering and $\deg B=\deg X$.
Then $X,Y,A,B$ form a good solution of $A\circ X=B\circ Y$
if and only if there exists a commutative diagram of
rational functions 
\[
\begin{CD}
\C\P^1 @>V>> \C\P^1 @>U>> \C\P^1\\
@VV X V @VV X V @VV B V\\
\C\P^1 @>C>> \C\P^1 @>D>> \C\P^1
\end{CD}
\]
such that $Y=U\circ V$, $A=D\circ C$, and the following
conditions hold:
\begin{enumerate}
\renewcommand{\labelenumi}{\textup{(\arabic{enumi})}}
    \item The map $B\circ U$ is a normalization of $B$.
    In particular, $U$ is a Galois covering.

    \item There exists an automorphism $\phi$ of $G_X$
    such that
    \[
    V\circ\mu=\phi(\mu)\circ V,\qquad \mu\in G_X.
    \]

    \item The groups $G_X$ and $G_U$ satisfy
    \[
    G_X\cap G_U=\{e\}.
    \]
\end{enumerate}
\et
\pr
The ``if'' part follows directly from Theorem~\ref{t6}.
To prove the ``only if'' part, suppose that $X,Y,A,B$
form a good solution. 
Applying Theorem~\ref{t6} we see that all the maps involved are rational functions. In particular, $W$ is a rational Galois
covering. By the classical classification of finite M\"obius groups,
the isomorphic groups $G_X$ and $G_W$ are conjugate.
Therefore, there exist M\"obius transformations
$\eta$ and $\nu$ such that
\[
W=\nu\circ X\circ\eta.
\] 
Thus, making suitable M\"obius changes of coordinates
in the source and target of $W$ and adjusting the
other maps accordingly, we obtain the required
diagram with $W=X$. \qed

We finish this subsection with the following result,
which describes rational solutions of \eqref{fe} 
in the case where, instead of $X$, the map $B$ is
a Galois covering.
Clearly, by extracting a common compositional right
factor of $X$ and $Y$, we can reduce to the case
$\C(X,Y)=\C(z)$, which we therefore assume.

\bt \l{t8}
Let $X,Y,A,B$ be rational functions such that $\C(X,Y)=\C(z)$. Assume that  $B$
is a Galois covering.
Then the equality $A\circ X=B\circ Y$ holds if and only
if $X$ is a Galois covering and there exist a M\"obius
transformation $\eta$ and rational functions $V,C,D$
such that the diagram
\[
\begin{CD}
\C\P^1 @>V>> \C\P^1 @>\eta>> \C\P^1\\
@VV X V @VV X V @VV B V\\
\C\P^1 @>C>> \C\P^1 @>D>> \C\P^1
\end{CD}
\]
commutes, $Y=\eta\circ V$, $A=D\circ C$, and $V$ is $G_X$-equivariant with respect to
an automorphism of $G_X$. Moreover, 
in this case,
$
n(A,B)=\deg D.
$
\et

\pr  The ``if'' part is obvious.
To prove the ``only if'' part, suppose that
\eqref{fe} holds. 
By Theorem~\ref{bgal}, $X$ is a Galois covering.
Moreover, Theorem~\ref{phi} gives an injective homomorphism
\[
\phi:G_X\longrightarrow G_B
\]
such that
\[
Y\circ\mu=\phi(\mu)\circ Y,\qquad \mu\in G_X.
\]
Since the finite groups $G_X$ and $\phi(G_X)$ are
isomorphic, there exists a M\"obius transformation
$\eta$ such that
\be\l{eta} 
\eta^{-1}\phi(G_X)\eta=G_X, 
\ee
and, setting $V=\eta^{-1}\circ Y$, we obtain
\[
V\circ\mu=
(\eta^{-1}\circ\phi(\mu)\circ\eta)\circ V,
\qquad \mu\in G_X.
\]
Thus, $V$ is $G_X$-equivariant.

The function $B\circ\eta$ is invariant under
$G_X$ by \eqref{eta}, and $X\circ V$ is invariant
under $G_X$ by the equivariance of $V$.  
Since $X$ is a Galois covering, this implies that there exist
rational functions $C$ and $D$ such that
\be \l{giv} 
X\circ V=C\circ X,\qquad B\circ\eta=D\circ X.
\ee 
Consequently,
\[
A\circ X=B\circ Y
=B\circ\eta\circ V
=D\circ X\circ V
=D\circ C\circ X,
\]
so $A=D\circ C$.
This gives the required diagram.

Finally, \eqref{giv}  gives
\[
\deg D=\frac{\deg B}{\deg X}.
\]
By Theorem~\ref{phi}, we also have
\[
n(A,B)=\frac{\deg B}{\deg X}.
\]
This proves the required degree formula. \qed

\subsection{\l{ex} Some examples} 
In this section we show that for the rational function
\[
X=z+\frac{1}{z}
\]
there exist infinitely many indecomposable rational
functions $Y$ satisfying
\be \label{ip} \C(X,Y)=\C(z),\qquad 
\mathbb{C}(X)\cap\mathbb{C}(Y)\neq \mathbb{C}
\ee
which are neither $G_X$-equivariant 
nor Galois coverings.
Thus, these examples show that both the maximality
condition in Theorem~\ref{t4} and the condition involving
critical points in Theorem~\ref{t22} are essential.

Let us set 
\[
Y=\frac{1-z^l}{z^{l+m}-1}, \qquad
B=z^l(z+1)^m,
\]
where $l$ and $m$ are positive integers. 
Clearly, for the  function 
\[
(B\circ Y)(z)
=
z^{lm}
\left(\frac{1-z^l}{z^{l+m}-1}\right)^l
\left(\frac{z^m-1}{z^{l+m}-1}\right)^m
\]
the equality 
\be \l{ref} 
(B\circ Y)\circ \frac{1}{z}= B\circ Y
\ee
holds. Since $\frac{1}{z}$-invariant rational functions form a subfield of $\mathbb{C}(z)$ generated by $X$, this implies that there exists a rational function $A$ such that 
\be \label{ss}
A\circ X=B\circ Y.
\ee 
Thus, for any choice of parameters, the functions
\[
X = z + \frac{1}{z}, \qquad Y = \frac{1 - z^l}{z^{l+m} - 1}
\]
satisfy the second condition in \eqref{ip}. On the other hand, since $X$ is indecomposable,
the first condition follows from
\be \l{plm}
Y\circ\frac{1}{z}=z^m\cdot Y,
\ee
which shows that $Y\notin\C(X)$.

Let us now show that, for infinitely many values
of the parameters $l$ and $m$, the function $Y$ is
indecomposable, yet neither $G_X$-equivariant nor
a Galois covering. Indeed, setting $l=1$ and $m=p$,
where $p\geq3$ is a prime, we see that

\begin{equation}\label{gi}
Y = \frac{1-z}{z^{p+1}-1} = -\frac{1}{1 + z + \cdots + z^{p}}.
\end{equation} 
Thus, $\deg Y = p$, implying that $Y$ is indecomposable.   

Furthermore, $Y$ is not a Galois covering, since its deck transformation group $G_Y$ is trivial.  
Indeed, any element $\delta \in G_Y$ permutes the points in each fiber of $Y$. 
Since $\infty$ is the unique zero of $Y$, the map $\delta$ fixes $\infty$ and is therefore affine.
Since $\delta$ permutes $Y^{-1}(\infty)$ and two circles
cannot have more than two common points unless they
coincide, we see that $\delta$ preserves the unit circle
$S^1$, and thus has the form $\delta(z)=cz$, where $|c|=1$. 
Nevertheless, one can easily see that for such $\delta$ the equality $Y \circ \delta = Y$ is impossible unless $c=1$.

Finally, $Y$ is not $G_X$-equivariant. Indeed, by \eqref{plm},  
the equality
\[
Y \circ \frac{1}{z} = \delta \circ Y,
\]
where $\delta$ is a M\"obius transformation,
would imply that
\[
z^p \cdot Y = \delta \circ Y.
\] 
Hence,
$
z^p = \mu \circ Y
$ 
for the rational function $\mu(z)=\delta(z)/z$. In view of the equality of degrees
on both sides, $\mu$ must be a M\"obius transformation. However,
this would imply that $G_Y$ is a cyclic group of order $p$,
which contradicts the argument presented above.

Let us also observe that 
\be \l{l} 
\C(X)\cap\C(Y)=\C(H),
\ee
where $H=B\circ Y$, 
and hence
\[
[\C(z):\C(X)\cap\C(Y)]=p(p+1).
\]
To prove this, it is enough to show that 
 $B=z(z+1)^p$
is indecomposable.
Indeed, since \[
\C(H)\subseteq\C(X)\cap\C(Y)\subseteq\C(Y),
\]
in this case $\C(X)\cap\C(Y)$ 
must equal one of the endpoint fields.
It cannot equal $\C(Y)$, since $Y\notin\C(X)$, and hence \eqref{l} holds. 

To finish the proof, we observe that the local degree
of $B$ at $-1$ is the prime $p$.
Thus, any non-trivial decomposition of $B$
would have a factor of degree at least $p$, implying
$\deg B\geq2p>p+1$, a contradiction.

In conclusion, let us mention that, for $X$ of degree two,
the problem of describing rational functions $Y$
satisfying \eqref{ip} is closely related to the problem
of describing non-diagonal solutions of the functional
equation
\begin{equation}\label{so1}
B\circ Y=B\circ Z
\end{equation}
in rational functions.

Indeed, if $X$ is a rational function of degree two
and $\mu$ is its non-identity deck transformation,
then equality \eqref{ss} implies that
\begin{equation}\label{so2}
B\circ Y=B\circ(Y\circ\mu),
\end{equation}
providing a non-diagonal solution of \eqref{so1}.

On the other hand, if the algebraic curve
\be \l{curva}
\frac{B(x)-B(y)}{x-y}=0
\ee
is irreducible, then by \cite[Theorem~3.1]{ent},
any non-diagonal rational solution of \eqref{so1}
satisfying $\C(Y,Z)=\C(z)$ has the form
$Z=Y\circ\mu$ for some involution
$\mu\in\Aut(\C\P^1)$.
Thus, $B\circ Y$ is $\mu$-invariant.
Taking $X$ to be the quotient map for the group
$\{e,\mu\}$, we obtain a rational function $A$
such that \eqref{ss} holds.
Since $Y\circ\mu\ne Y$, the resulting functions
$X$ and $Y$ satisfy \eqref{ip}.

The absence of non-diagonal rational or meromorphic
solutions of \eqref{so1} has applications to the
geometry and dynamics of $B$ (see, e.g., \cite{tame},
\cite{plo}). This makes the problem of describing
rational solutions of \eqref{so1} particularly
interesting. For some 
partial results, we refer the reader to \cite{az}, \cite{alg}, \cite{ent}, 
\cite{r4}, \cite{seg}.


\end{document}